\documentclass[a4paper,11pt]{article}
\usepackage[T1]{fontenc} 
\usepackage[cp1250]{inputenc}
\usepackage{lmodern}
\usepackage[symbol]{footmisc}
\usepackage[english]{babel}
\usepackage{latexsym}
\usepackage{amsmath}
\usepackage{caption}
\usepackage{subcaption}
\usepackage{amsthm}
\usepackage{amssymb}
\usepackage{indentfirst}
\usepackage{amsfonts}
\usepackage{stackrel}
\usepackage{float}
\usepackage{dsfont}
\usepackage{tikz}
\usepackage[mathscr]{eucal}
\usepackage{xcolor}

\newtheorem{theorem}{Theorem}
\newtheorem{lemma}[theorem]{Lemma}

\newtheorem{example}{Example}

\newtheorem{definition}{Definition}

\newcommand{\R}{\mathbb{R}}
\newcommand{\N}{\mathbb{N}}

\newcommand{\GG}{\mathbb{G}}
\newcommand{\EE}{\mathbb{E}}
\newcommand{\algo}{\mathfrak{a}}

\newcommand{\eps}{\varepsilon}

\usepackage{tocbibind}

\usepackage{hyperref}
\newcommand{\footremember}[2]{%
   \footnote{#2}
    \newcounter{#1}
    \setcounter{#1}{\value{footnote}}%
}
\newcommand{\footrecall}[1]{%
    \footnotemark[\value{#1}]%
}

\begin{document}
\begin{center}\textsc{\large Approximate solutions to the Travelling Salesperson Problem on semimetric graphs}
\\
\vspace{1cm}
{Mateusz Krukowski,\footremember{IMPL}{\textsc{\L\'od\'z University of Technology, \ Institute of Mathematics, \\ W\'ol\-cza\'n\-ska 215, \
  	90-924 \ \L\'od\'z, \ Poland}} Filip Turobo\'s\footrecall{IMPL}}
\end{center}

\begin{abstract}
With the aid of the relaxed polygonal inequality (introduced by Fagin et al.) we strive to extend the applicability of Christofides approximation technique to the scope of all complete finite weighted graphs with positive weights. First section acquaints the Reader with the class of semimetric graphs and proves that every finite graph admits $\gamma$-polygon structure. Sections 2 and 3 establish the necessary notions from the graph and optimization theory to tackle the Traveling Salesperson Problem. In section 4 the minimal spanning tree method is introduced, while section 5 focuses on the analysis of this method through the lens of $\gamma$-polygon graphs. The final section of the paper adjusts the technique of Christofides by obtaining $\frac{3\gamma}{2}-$approximation for the TSP.
\end{abstract}
\textbf{AMS: 54E25, 
90C27, 
68R10, 
05C45, 
}

\section{Semimetric initial framework}

In the past several decades we have witnessed the formation of close ties between discrete mathematics and the theory of metric spaces.\footnote{For a detailed discussion on the interconnection between these mathematical disciplines, see \cite[Preface]{Blumenthal1953}.} The underlying reason for this ``mathematical marriage'' is that metric spaces constitute an abstract (but still useful) setting for various problems in discrete mathematics (this fact has already been emphasized by Piotr Indyk and Nathan Linial).\footnote{See \cite{Indyk1999} and \cite{Linial2002}, respectively.} As the years went by mathematicians started to fiddle with the generalizations of metric spaces and their possible applications to the realm of discrete mathematics. Among the most fruitful generalizations one should list hemimetric spaces (also known in the literature as quasimetric spaces), where the symmetry axiom ``$d(x,y) = d(y,x)$'' in the definition of metric is skipped.\footnote{See \cite{Brattka2003}, \cite{Gustavsson1975}  and  \cite[p. 700]{Smyth1993}.} Such spaces emerge naturally when we measure time needed to traverse from point $A$ to point $B$ in the mountains (you need not be a highlander to know that the climb is usually longer than the descent).
Another very popular generalization of the metric space is the pseudometric space, where the first axiom of a metric, i.e.
$$d(x,y) = 0 \ \iff  x = y,$$
    
\noindent
need not hold.\footnote{See \cite[p. 1]{Howes1995}.} It turns out that pseudometric spaces are especially common in the domain of fuzzy sets.\footnote{See \cite{Voxman1981,Lowen1984}.}

The applications of hemimetric and pseudometric spaces in various branches of mathematics are rather encouraging to say the least. Following that clue, our paper focuses on another family of spaces, called the semimetric spaces:\footnote{The literature on semimetric spaces is quite vast and we recommend the following list of references as a well grounded starting point:\cite{Bessenyei2017,Ceder1961, Jachymski2018,Chrzaszcz2018,Gottlieb2017,Turobos2020,Pareek1972,Wilson1931}.}
\begin{definition}
For a nonempty set $X$, a function $d:X\times X \longrightarrow [0,+\infty)$ is called a semimetric if it satisfies the following two conditions:
\begin{itemize}
	\item $\forall_{x,y\in X}\ d(x,y) = 0 \ \iff \ x=y$,
	\item $\forall_{x,y\in X}\ d(x,y) = d(y,x)$.
\end{itemize}

\noindent
The pair $(X,d)$ is called a semimetric space.
\label{semimetricdefinition}
\end{definition}		

Browsing through the literature one encounters multiple sources which assume that a semimetric space is first-countable.\footnote{See both articles \cite{Cook1977,Lutzer1971} or Definition 9.5 and Theorem 9.6 in the monograph \cite{Kunen1984}.} Throughout the paper we are only concerned with finite spaces (finite graphs to be precise), so the first-countability condition is satisfied, although not explicitely required in Definition \ref{semimetricdefinition}.

Amongst all semimetric spaces it is convenient to distinguish those, which have some counterpart of the celebrated triangle inequality. For our purposes, let us introduce $\beta-$metric spaces and $\gamma-$polygon spaces:\footnote{Both $\beta-$metric and $\gamma-$polygon spaces are established  in the literature: \cite{An2015,Andreae1995,Bandelt1991,Bakthin1989,Bourbaki1966,Jachymski2018,Chrzaszcz2018,Dung2016,Fagin2003,Turobos2020,Paluszynski2009,Schroeder2006,Suzuki,Wilson1931,Xia2009} serve just as a couple of examples. }

\begin{definition}
A semimetric space $(X,d)$ is said to be:
\begin{itemize}
    \item $\beta-$metric space if $\beta\geqslant 1$ is the smallest number such that the semimetric $d$ satisfies the $\beta-$triangle inequality:
    \begin{gather}
    \forall_{x,y,z\in X}\ d(x,z)\leqslant \beta ( d(x,y) + d(y,z) ),
    \label{betainequality}
    \end{gather}
    \item $\gamma-$polygon space if $\gamma\geqslant 1$ is the smallest number such that the semimetric $d$ satisfies the $\gamma-$polygon inequality:
    \begin{gather}
    \forall_{n\in\N} \ \forall_{x_1,\dots,x_n\in X}\ d(x_1,x_n)\leqslant \gamma\cdot \sum_{k=1}^{n-1} d(x_k,x_{k+1}).
    \label{polygonalinequality}
    \end{gather}
\end{itemize}
\end{definition}

It goes without saying that every $\gamma-$polygon space is automatically a $\beta-$metric space with $\beta \leqslant \gamma$ (we will shortly provide an example that the constants need no coincide). Conversely, if we assume that $X$ is finite (as we do in this paper), then every $\beta-$metric space is a $\gamma-$polygon space with $\gamma\leqslant \beta^{|X|-1}.$ A closer inspection\footnote{Carried out by Suzuki in \cite{Suzuki}.} reveals that we have an estimate $\gamma \leqslant \beta^{\lceil\log_2(|X|-1)\rceil}$ (assuming that $X$ consists of more than one point).


It is a relatively easy, though still useful remark that every finite semimetric space admits both $\beta$-metric and $\gamma$-polygon structure:\footnote{As far as we know this observation first appeared in \cite{Chrzaszcz2018}.}

\begin{theorem}\label{lamma}
Every finite semimetric space $(X,d)$ is a $\gamma-$polygon space with 
\begin{equation}
    \gamma = \sup\left\{ \frac{d(x_1,x_n)}{\sum_{k=1}^{n-1} d(x_k,x_{k+1})} \ : n\in \N, \ x_1,\dots ,x_n\in X,\ x_1\neq x_n \right\}. \label{formulaforgamma}
\end{equation}

\noindent
In particular, every finite semimetric space $(X,d)$ is a $\beta-$metric space with
\begin{equation}
\beta = \sup\left\{ \frac{d(x,z)}{d(x,y) + d(y,z)} \ : \ x,y,z\in X,\ x\neq z \right\}. \label{formulaforbeta}
\end{equation}
\end{theorem}
\begin{proof}
Without loss of generality, we can assume that the space $X$ contains at least $2$ elements (otherwise $(X,d)$ is trivially a $\gamma-$polygon space for any $\gamma\geqslant 1$). Let
\[
s := \sup\left\{ \frac{d(x_1,x_n)}{\sum_{k=1}^{n-1} d(x_k,x_{k+1})} \ : n\in \N, \ x_1,\dots ,x_n\in X,\ x_1\neq x_n \right\}
\]
\noindent
and observe that it is a well-defined, finite number, which satisfies (for any choice of $x_1,x_2\in X$ with $x_1\neq x_2$)
$$s \geqslant \frac{d(x_1,x_2)}{d(x_1,x_2)} = 1.$$

\noindent
Furthermore, let us fix $n\in\N$ and elements $x_1,\dots,x_n\in X$ and consider two cases:
\begin{itemize}
    \item if $x_1 = x_n$ then 
    \[
    d(x_1,x_n) = 0 \leqslant s\cdot \sum_{k=1}^{n-1}\ d(x_k,x_{k+1});
    \]
    
    \item if $x_1\neq x_n$ then
    \begin{eqnarray*}
    d(x_1,d_n) &=& \frac{d(x_1,x_n)}{\sum_{k=1}^{n-1}\ d(x_k,x_{k+1})} \cdot \sum_{k=1}^{n-1}\ d(x_k,x_{k+1})
    \\
    &\leqslant & \sup\left\{ \frac{d(y_1,y_n)}{\sum_{k=1}^{n-1}\ d(y_k,y_{k+1})} \ : \ y_1,\dots ,y_n\in X,\ y_1\neq y_n \right\} \cdot \sum_{k=1}^{n-1}\ d(x_k,x_{k+1})
    \\
    &= & s \cdot \sum_{k=1}^{n-1}\ d(x_k,x_{k+1}).
    \end{eqnarray*}
\end{itemize}

\noindent
This proves that $\gamma \leqslant s.$

In order to prove the reverse inequality suppose, for the sake of argument, that there exists $\eps>0$ such that $\gamma = s-\eps$. From the definition of supremum there exists an $n\in \N$ as well as some elements $x_1,\dots,x_n \in X$ such that 
\[
\frac{d(x_1,x_n)}{\sum_{k=1}^{n-1}d(x_k,x_{k+1})} > s-\eps.
\]
Multiplying both sides of this inequality yields

\[
d(x_1,x_n) > (s-\eps)\cdot \sum_{k=1}^{n-1}d(x_k,x_{k+1}),
\]
which contradicts the assumption that $\gamma < s$. This concludes the first part of the theorem (for $\gamma-$polygon space). In order to prove the second part (a finite semimetric space is a $\beta-$metric space) we follow an analogous reasoning as the one presented above, fixing $n=3$.
\end{proof}

At this point we know that given a finite semimetric space it has to be both $\beta-$metric and $\gamma-$polygon space at the same time. Moreover, we also know that $\beta\leqslant \gamma,$ but in general, these constants may differ. To illustrate that point let us discuss the following example:

\begin{example}
Let $X:=\{1,2,3,4,5\}$ and let $d: X\times X \longrightarrow [0,+\infty)$ be a semimetric given by
\begin{equation}
    d(x,y):=\begin{cases}
    0,& x=y,\\
    1,& |x-y|=1, \\
    20,& |x-y|=4,\\
    4,& \text{otherwise.}
    \end{cases}
\end{equation}
See the following figure for the illustration of this space:

\begin{figure}[H]
    \centering
    \includegraphics[height=5cm]{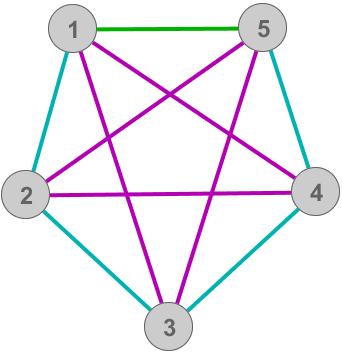}
    \caption{Blue edges denote the semidistance of length $1$, purple edges are of length $4$, while the only green edge is the longest and has the length $20$.}
    \label{fig:some_space}
\end{figure}

A simple calculation reveals that $(X,d)$ is a $4-$metric space, but it is not a $4-$polygon space since for a finite sequence $(x_k)_{k=1}^5$, $x_k:=k$ the constant $\gamma$ has to satisfy:
\[
d(x_1,x_5) =20 \leqslant \gamma \cdot\sum_{k=1}^4 d(x_k,x_{k+1}) = 4\gamma. 
\]

\noindent
This proves that $\gamma\geqslant 5.$ After a few more easy computations one can be convinced that $\gamma=5.$
\end{example}


From a theoretical standpoint, the existence of constants $\beta$ and $\gamma$ suffices to construct an ``abstract mathematical theory''. However, in practice it is preferable to know the precise values of these constants, hence the need for an algorithm enabling the computation of $\beta$ and $\gamma$ in a reasonable time-scope. We will come back to this issue once we discuss optimization and complexity theory in Section \ref{section:optimizationandcomplexitytheory}.

\section{Graph theory through the lenses of semimetric spaces}

We have already defined the elementary concepts of semimetric theory. It is time we invoked the crucibles of graph theory and investigate their interplay with the notions of $\beta-$metric and $\gamma-$polygon spaces intoduced earlier. An additional advantage of our concise review is that we lay down the notational convention used in the sequel. We commence with the definition of a graph:\footnote{The definition of a graph is based on \cite[p. 2]{Diestel2000}, whereas the definition of a weighted graph was taken from \cite[p. 463]{Fletcher1991}.}

\begin{definition}
A pair $G := (V,E)$ is called a graph if $V$ is a finite, nonempty set and 
$$E \subset \{\{x,y\} \ : \ x,y\in V,\ x\neq y \}.$$ 

\noindent
The elements of $V$ and $E$ are called vertices and edges, respectively.

A weighted graph is a pair $(G,\omega)$, where $G = (V,E)$ is a graph, $E\neq \emptyset$ and $\omega : E \longrightarrow (0,+\infty)$ is a positive function, called the weight.
\end{definition}

A couple of remarks concerning this definition are in order. First off, it is usually convenient to utter the phrase ``\textit{Let $G$ be a graph}'' and then refer to the vertex and edge set of $G$ as $V(G)$ and $E(G)$, respectively. This is a convention that we adhere to in the sequel.\footnote{Identical convention can be found in \cite{Diestel2000}.} Furthermore, those acquainted with the graph terminology will surely recognize our graphs to be \textit{undirected} and \textit{simple}. ``\textit{Undirectedness}'' of a graph means that the edges are not oriented, i.e. every edge is a set $\{x,y\}$ rather than an ordered pair $(x,y).$ On the other hand ``\textit{simplicity}'' means that the graph contains no ``\textit{loops}'' (i.e. edges of the form $\{x,x\}$) or \textit{multiedges} (i.e. $E$ is a set and not a multiset). The need for \textit{multiedges} and \textit{multigraphs} will arise, however, towards the end of the paper, hence we take the liberty of including the formal definition of these objects:

\begin{definition}
A pair $\GG :=(V,\EE)$ is called a multigraph if $V$ is a finite, nonempty set and
\[
\EE\subset \bigg\{ (k,\{x,y\})\ : \ x,y\in V,\ x\neq y,\ k\in\N_0 \bigg\}.
\]
The elements of $\EE$ are called multiedges (while the elements of $V$ are still called vertices).
A weighted multigraph is a pair $(\GG,\omega)$, where $\GG= (V,\EE)$ is a multigraph, $\EE\neq \emptyset$ and $\omega:\EE\longrightarrow (0,+\infty)$ is a positive function, called the weight.
\end{definition}

It goes without saying that the index $k$ in the above definition is used to distinguish different edges connecting the same pair of vertices. More importantly, the way we define a weight (on a graph or a multigraph) slightly deviates from the ones existent in the literature - this is because we assume $\omega$ to be a positive function.\footnote{For instance, Bondy and Murty in \cite[p. 50]{Bondy} allow $\omega$ to be a real-valued function while Bl\"aser in \cite{Blaser2008} uses only non-negative rationals as edge weights.}

Having defined what we mean by the term ``graph'' let us proceed with a series of related concepts:\footnote{See \cite{Bondy,Diestel2000,Gross}.}

\begin{definition}\label{def:subgraph etc}
Let $G$ be a graph.
\begin{itemize}
    \item Graph $F$ is called a subgraph of $G$ if $V(F)\subset V(G)$ and $E(F)\subset E(G)$. We denote this situation by writing $F\subset G$.
    \item Graph $P$ is called a path if vertices of $P$ can be arranged in a sequence so that two vertices are adjacent if and only if they are consecutive in this sequence.
    \item Graph $C$ is called a cycle if the removal of any edge in $C$ turns it into a path. If $G$ does not contain any cycle as a subgraph, then it is called a acyclic graph.
    \item Subgraph $H\subset G$ is called a Hamiltonian path/cycle if it is a path/cycle visiting each vertex of $G$ exactly once. The family of all Hamiltonian cycles in $G$ is denoted by $\mathfrak{h}(G).$
\end{itemize}
\end{definition}

A graph $G$ whose family of Hamiltonian cycles is nonempty is sometimes called a \textit{Hamiltonian} graph. The definition of path allows us to introduce one more important notion.

\begin{definition}
A graph $G$ is said to be connected if for any pair of vertices $x,y\in V(G)$ there exists a path $P\subset G$ such that $x,y\in V(P)$.
\end{definition}

The notion of a subgraph introduced in Definition \ref{def:subgraph etc} is inherently independent of the weight function (if such exists) on graph $G$. However, if $\omega$ is a weight on $G$ and $F$ is a subgraph of $G$, then $(F,\omega|_{E(F)})$ is a weighted graph and $\omega|_{E(F)}$ is called an \textit{induced weight}. Furthermore, it is often convenient to speak of a weight of a subgraph (or the whole graph itself), which we define as 
\begin{equation*}
\omega(F):=\sum_{e\in E(F)} \omega(e).
\end{equation*}

It is hard to deny that this definition is a slight abuse of notation - after all, we use the same symbol ``$\omega$'' to weigh both edges and a subgraph. Formally this is a mistake since edges and subgraphs are objects from different ``categories'' - edges are unordered pairs of elements while subgraphs are pairs of vertex and edge sets. However, we believe that such a small inconsistency should not lead to any kind of misunderstanding (one reason being that we use capital letters for graphs and subgraphs and minuscule letters for edges).

We have now gathered all the necessary ingredients to formulate the \textit{Travelling Salesperson Problem} (or TSP for short), which is the focal point of our paper:\footnote{It should be emphasized that there are multiple other ways to define this problem, for example as integer programming problem with constraints on vertex degrees, e.g. \cite{Danzig1959,Laporte1992,Matai2010,Miller1960}.}\\
\vspace{0.01cm}\\
\fbox{%
\parbox{0.95\textwidth}{\large
\textit{For a weighted graph $(G,\omega)$ with $\mathfrak{h}(G)\neq\emptyset$ find a Hamiltonian cycle $H$ with minimal weight.}
}
}
\\\vspace{0.01cm}\\

Travelling salesperson problem is one of the most famous mathematical puzzles\footnote{Timothy Lanzone even directed a movie ``\textit{Travelling Salesman}'', which premiered at the International House in Philadelphia on June 16, 2012. The thriller won multiple awards at Silicon Valley Film Festival and New York City International Film Festival the same year.} and a detailed account of its history is far beyond the scope of this paper. Instead, we recommend just a handful of sources - for an indepth discussion on both history and possible attempts at solving this problem see: \cite{Applegate2006,Cook2012,Fleischmann1988,Fonlupt1992,Gutin2001,Laporte1992,Reinelt1994,Snyder2019,Van2020}.

Our formulation of the TSP requires the weighted graph $G$ to satisfy $\mathfrak{h}(G)\neq\emptyset.$ This is fairly obvious, because if there is no Hamiltonian cycle it is impossible to look for one with minimal weight! That being said, in the sequel we will usually consider the TSP on a complete weighted graphs, which trivially satisfy the condition $\mathfrak{h}(G)\neq\emptyset.$ The reason we may restrict our attention to complete graphs is that it is always possible to define $\omega(\{x,y\})$ as a ``sufficiently large value'' for edges $\{x,y\}$ which are not in the original graph $G$. A formal result which encapsulates this intuition is provided below:

\begin{theorem}(well-known in mathematical folklore)\\
Let $(G,\omega_G)$ be a graph with $n:=|V(G)|\geqslant 3$ vertices. There exists a complete weighted graph $(K_{n}, \omega)$ such that:
\begin{enumerate}
\item $G$ is a subgraph of $K_n$ and $\omega|_G = \omega_G$.
\item If $H$ is a minimal Hamiltonian cycle in $(K_n,\omega)$ and $\omega$ is a subgraph extension of $\omega_G$, then exactly one of the following holds:
\begin{itemize}
	\item $ \omega (H) \leqslant \omega(G)$ and $H$ is a minimal Hamiltonian cycle in $G$ as well, or
	\item $ \omega (H) > \omega(G)$ and $G$ does not admit any Hamiltonian cycle. 
\end{itemize}
\end{enumerate}
\end{theorem}
\begin{proof}
Let $K_n$ be a complete graph on $n$ vertices of the graph $G$ and let the weight $\omega$ on $K_n$ be defined with the formula:
\begin{equation}\label{eqn:4}
 \omega(\{x,y\}) := \begin{cases}
\omega_G(\{x,y\}), & \text{if } \{x,y\}\in E(G),\\
\sum_{e\in E(G)}\ \omega_G(e), & \text{otherwise. }\\
\end{cases}
\end{equation}

\noindent
Immediately it follows that $G$ is a subgraph of $K_n$ and $\omega|_G =\omega_G,$ which is the first part of the theorem.

For the second part let $\omega$ be a subgraph extension of $\omega$ and let $H$ be a minimal Hamiltonian cycle in $(K_n,\omega).$ We divide our reasoning into two separate cases:
\begin{itemize}
    \item Suppose that $\omega(H)\leqslant \omega(G).$ Definition (\ref{eqn:4}) implies that $H$ cannot contain any edge which is not in $E(G),$ since otherwise we would have 
    $$\omega(H) > \omega(G) = \sum_{e\in E(G)}\ \omega(e).$$
    
    \noindent
    Furthermore, there cannot exist a Hamiltonian cycle $H'$ on $(G,\omega_G)$ such that $\omega(H') < \omega_G(H),$ simply because that would prevent $H$ from being a minimal Hamiltonian cycle on $(K_n,\omega).$ Therefore we have to conclude that $H$ is a minimal Hamiltonian cycle on $(G,\omega_G).$
    
    \item Suppose that $\omega(H)>\omega(G).$ If there existed a minimal Hamiltonian cycle $H'$ on $G$ then we would have 
    $$\omega(H') \leqslant \omega(G) < \omega(H).$$
    
    \noindent
    This contradicts the fact that $H$ is a minimal Hamiltonian cycle in $K_n$ (since there is a cycle with smaller weight, namely $H'$), which concludes the proof.  
\end{itemize}
\end{proof}


In the light of the above, we may restrict our focus to the case where the graph in the formulation of the TSP is a complete weighted graph 
 i.e., $G=(K_n,\omega)$. For such graphs we are able to impose the semimetric structure on $V(G)$ as follows:\footnote{Clearly, the function defined in \eqref{semimetriconwholegraph} is both symmetric (due to the undirected nature of the graph) and satisfies $d(x,y)=0\iff x=y$ for all $x,y\in V$.}
\begin{equation}\label{semimetriconwholegraph}
d(x,y):=\begin{cases} 
\omega\left(\{x,y\}\right),& \text{ where }x\neq y;\\
0,& \text{ where } x=y.
\end{cases}
\end{equation}

This allows us to define some particular subclasses of complete weighted graphs as follows:

\begin{definition}
A graph $G=(K_n,\omega )$ is called
\begin{itemize}
    \item a metric graph if $d$ is a metric on $V(G)$;
    \item a $\beta-$metric graph if $d$ satisfies \eqref{betainequality};
    \item a $\gamma-$metric graph if $d$ satisfies \eqref{polygonalinequality}.
\end{itemize}
\end{definition}

Depending on the additional conditions imposed on the graph weights, we can obtain more specific variants of TSP. In particular, when $d$ satisfies the triangle inequality, we can talk about the metric TSP (cases where particular instance is considered on a plane are sometimes referred to as Euclidean Traveling Salesman Problem). While it might seem the metric TSP should be much easier to solve, it was shown by Christos Papadimitriou that even in the case of Euclidean metric setting, the Traveling Salesman Problem remains NP-hard.\footnote{See \cite{Papadimitriou1977} for the proof.} It is therefore reasonable to consider approximation algorithms as a valuable asset in the array of techniques designed for coping with such problems.\footnote{It is worth noting that multitude of exact algorithms for TSPs which work in exponential time are known (for survey of branch-and-bound methods see \cite{Balas1985}, several other types of methods are discussed in \cite{Fischetti2007} -- for more detailed approach see also the references within these surveys). but they are out of the scope of this paper and they will not be discussed in the latter part of the article. Since reasonable approximations (which are often within $1-2\%$ error margin from the optimal solution)   can be found in fractions of the time needed to perform the exact-search, we will focus solely on these techniques.}

\section{Optimization and complexity theory}
\label{section:optimizationandcomplexitytheory}

We believe that our paper might be of interest to both researchers in algorithm-related branches as well as those working wiht various generalizations of metric spaces. Bearing that in mind, the focal point of the current subsection is a brief summary of fundamental notions and facts in optimization and complexity theory. 

We start off with the idea of ``\textit{optimization problem}'' itself:\footnote{We are aware that some authors define the optimization problem otherwise, see e.g. \cite[Section 1.4]{Ausiello1999}. For a more advanced discussion of optimization problems we highly recommend \cite[Section 1.1.1]{Andreasson} and \cite[Sections 1.1 and 1.3]{GlobalOptimization}.}
\begin{definition}
Let $X$ be a nonempty, finite set and let $f:X\longrightarrow (0,+\infty)$ be an arbitrary function. A problem of finding an element $x^* \in X$ which satisfies
\begin{equation}
f(x^*)=\min_{x\in X} f(x),\label{minimality_condition}
\end{equation}

\noindent
is called an optimization problem. 
\end{definition}

Set $X$ and function $f$ in the optimization problem are called the \textit{search space} and the \textit{target function}, respectively. Furthermore, every element satisfying (\ref{minimality_condition}) is said to be an \textit{optimal solution} (there might be more than one such element), while the elements of $X$ are referred to as \textit{feasible solutions}.

At first glance, solving a optimization problem on a finite set seems trivial. Why not go through all elements of $X$ one by one, checking the values $f(x)$ and choosing the minimal one? Well, in theory this is a perfectly valid approach but in practice it is too cost-inefficient (takes too much time and memory space). Take, for instance, the Travelling Salesperson Problem, which is an optimization problem - for a given weighted graph $(G,\omega)$ it suffices to put $X := \mathfrak{h}(G)$ and $f:=\omega$. If $G = K_{61}$ then there are $\frac{1}{2}\cdot 60!$ Hamiltonian cycles,\footnote{In general, for a complete graph $K_n$ there are $\frac{1}{2}\cdot (n-1)!$ Hamiltonian cycles. This is because there are $n!$ possible arrangments of vertices $x_1,\dots,x_n$ and each such arrangement describes a Hamiltonian cycle. However, this description is not unique, for exmaple the sequence $(x_1,\dots,x_n)$ describes the same Hamiltonian cycle as $(x_2,\dots,x_n,x_1),\ (x_3,\dots,x_n,x_1,x_2)$ and so on. Furthermore, the ``inverted'' sequences $(x_n,\dots,x_1),\ (x_1,x_n,\dots,x_2),\ (x_2,x_1,x_n,\dots,x_3)$ also describe the same Hamiltonian cycle, so in fact there are always $2n$ descriptions for the same cycle. Hence, the total number of possible arrangements (that is $n!$) has to be divided by $2n$, which yields the final answer of $\frac{1}{2}\cdot (n-1)!$.} which is more than the number of atoms in the observable universe (roughly $10^{80}$ particles)! Without divine intervention solving such a gargantuan problem is simply inconceivable. 

Let us pause for a moment and see how the tables have turned. What seemed like an innocuous task now appears as an insurmountable hinderance. Luckily, not all is lost. There is a way out but we need to allow for a small ``hustle'': instead of searching for the optimal solution, we look for an \textit{approximate} solution, which is (in some sense) ``good enough''. More than that, we usually look for an approximation \textit{method} (rather than a single approximate solution for a fixed optimization problem), i.e. an algorithm which can be applied to a whole family of optimization problems:\footnote{An interested Reader is invited to compare our definition with the one of approximation algorithm given by Ausiello et al. \cite[Def. 3.5]{Ausiello1999}. Another clever approach can be found in Crescenzi, Vigo et al. \cite[Def. 2]{Crescenzi1999}.} 

\begin{definition}
Let $s\geqslant 1$ and let $\mathcal{P}$ be a set of indices for a family of optimization problems with target functions $f_p : X_p\longrightarrow (0,+\infty).$ A function $\algo:\mathcal{P}\to \bigcup_{p\in \mathcal{P}} X_p$ is called an $s$-optimal approximation method for $\mathcal{P}$ if for every $p\in\mathcal{P}$ we have
\begin{gather*}
    \algo(p) \in X_p\hspace{0.4cm}\text{and}\hspace{0.4cm} \frac{f_p\left(\algo(p)\right)}{f_p\left(x^*_p\right)} \leqslant s,
\end{gather*}

\noindent
where $x_p^*$ is some optimal solution of the $p-$th optimization problem.
\end{definition}

With such terminology at our disposal, we can say that the rest of the paper is devoted to finding $s$-optimal approximation methods for the Travelling Salesman Problem. We will focus on two methods:
\begin{itemize}
    \item Minimal spanning tree method on metric and $\gamma-$polygon graphs (Section \ref{doubleminimalspanningtreemethod});
    \item Christofides algorithm on $\gamma-$polygon graphs (Section \ref{christofides section}).
\end{itemize}

Before we proceed with the analysis of these approximation methods, let us close the current section with a couple of words regarding the basic principles of time complexity of algorithms. Please bear in mind, however, that our exposition is extremely rudimentary - we refer an inquisitve Reader to the vast literature on the subject, see the books of Oded Goldreich \cite[Section 1.2.3.4]{Goldreich2008} and \cite[Section 1.3.5]{Goldreich2010}, as well as \cite{Arora2009,Ausiello1999,Du2011,Papadimitriou1998}.

Let $\mathcal{P}$ be a set of indices for a family of optimization problems and let $\mathcal{M}:\mathcal{P}\longrightarrow \N$ be a function, which to every $p\in\mathcal{P}$ assigns the number of memory units (within the scope of the discussed model of computation, say the traditional Turing machine) used to describe the $p-$th problem. We say that an algorithm $\algo:\mathcal{P}\longrightarrow \R$ has computational (time) complexity  at most $F_{\algo}:\N \longrightarrow\N$ if there exists a constant $\alpha\geqslant 0$ such that for $p-$th problem the algorithm $\algo$ requires no more than $\alpha \cdot F_{\algo}(\mathcal{M}(p))$ units of computation. Whenever this holds, we say that $\algo$ can be executed in time $\mathcal{O}(F_\algo (n)).$\footnote{This definition is equivalent to the one introduced in \cite{Cormen2009}.}

We are now in position to go back to the problem that we left off after Theorem \ref{lamma} concerning the computation of constant $\beta$ for a finite semimetric space. A naive brute-force approach which calculates every quotient of the form $\frac{d(x,z)}{d(x,y) + d(y,z)}$ and picks the one with the heighest value has time complexity $\mathcal{O}(n^3)$, where $n$ is the cardinality of the semimetric space $X$. The problem of finding $\gamma$ is slightly more nuanced, but (surprisingly) has the same time-complexity:

\begin{theorem}
Let $(X,d)$ be a finite semimetric space consisting of $n$ distinct points.
The constant $\gamma$ such that $(X,d)$ satisfies the $\gamma$-polygon inequality can be found with time-complexity $\mathcal{O}(n^3)$.
\label{calculating_gamma}
\end{theorem}
\begin{proof}
Let $G:=(K_n,\omega)$ be a complete weighted graph, where $\omega\left(\{x,y\}\right):=d(x,y)$ for all pairs of distinct $x,y\in V(K_n)=X.$ On this graph we execute the Floyd-Warshall shortest paths algorithm,\footnote{See \cite[Chapter 25.2]{Cormen2009} for a thorough description along with a broad commentary, as well as the original paper of Floyd \cite{Floyd1962}.} which returns a function $D: X\times X \longrightarrow [0,\infty),$ where $D(x,y)$ denotes the length of the shortest path between points $x$ and $y$, i.e.

\[
D(x,y):=\min \left\{ \sum_{k=1}^{n-1} d(x_k,x_{k+1}) \ :\ n\in \N,\ x_1 = x,\ x_n = y,\ x_k \in X \text{ for }k \leqslant n\right\}.
\]

\noindent
The execution of the algorithm has time complexity $\mathcal{O}(n^3)$.  

Next, we define a function $\Gamma : \{(x,y)\ :\ x,y\in X,\ x\neq y\}\longrightarrow [0,\infty)$ by the formula 
\[
\Gamma(x,y) := \frac{d(x,y)}{D(x,y)}.
\]

\noindent
Since each edge of $G$ has positive weight, then $\Gamma$ is well-defined. Furthermore, $\Gamma(x,y)\geqslant 1$ since the shortest path between any distinct elements $x$ and $y$ is no longer than the edge $\{x,y\}$ - in other words it must be the case that $D(x,y)\leqslant d(x,y)$.

Finally, we observe that 
$$\gamma = \max_{(x,y)}\ \Gamma(x,y).$$

\noindent
Computationally, in order to find the maximum on the right-hand side of the equation we need to check $\frac{n(n-1)}{2}$ values of $\Gamma$ (due to the symmetry $\Gamma(x,y) = \Gamma(y,x)$). Therefore, the time-complexity of the whole procedure (including the Floyd-Warshall algorithm) is $\mathcal{O}(n^3).$\footnote{This time-complexity is estimated without using any kind of parallel computations, but can be further improved by employing these techniques (for example via using parallel Dijkstra algorithm).} 
\end{proof}

\section{Minimal spanning tree method on metric, $\beta-$metric and $\gamma-$polygon graphs}
\label{doubleminimalspanningtreemethod}

As far as we are aware, the first application of $\beta$-metric spaces to combinatorial optimization problems dates back to 1994 and the paper by Bandelt, Crama and Spieksma \cite{Bandelt1991}, where the authors use the relaxed version of the triangle inequality to construct approximation algorithms for multi-dimensional assignment problems with decomposable costs.\footnote{For modern applications in other combinatorial problems see e.g. \cite{Chen2018,Chen2020}.} The idea of using $\beta-$metric spaces swiftly gained recognition and it wasn't long before consecutive authors applied $\beta-$metric inequality to the Travelling Salesperson Problem - see \cite{Andreae1995,Andreae2001,Bender2000,Mohan2017,Momke2015}. On the hand, to the best of our knowledge, there are no known results in the literature which employ the $\gamma-$polygon inequality in the context of the Travelling Salesperson Problem.  

The crucial concept with which we will be preoccupied throughout this section is that of a minimal spanning tree:

\begin{definition}
Let $G$ be a graph.
\begin{itemize}
    \item A connected, acyclic subgraph $T\subset G$ is called a tree.
    \item A tree $T$ which satisfies $V(T)=V(G)$ is called a spanning tree of $G$.
    \item The family of all spanning trees of $G$ is denoted by $\mathcal{S}(G)$.
\end{itemize}

\noindent
If $(G,\omega)$ is a weighted graph, then any tree $T\in\mathcal{S}(G)$ which satisfies
\[
\omega(T) = \min_{T^\prime \in \mathcal{S}(G)} \omega(T^\prime)
\]
\noindent
is called a minimal spanning tree (or MST for short).
\end{definition}

Before we proceed with laying the basic outline of the minimal spanning tree method method let us introduce the concept of a tree traversal . This notion will significantly facilitate our future discussion:\footnote{The definition is based on \cite[p.10]{Diestel2000}.}

\begin{definition}
Let $G$ be a graph. 
\begin{itemize}
    \item An $n$-element sequence of vertices $(x_1,\dots,x_n)$ such that $\{x_k,x_{k+1}\}\in E(G)$ for every $k<n$ is called a walk (on graph $G$).
    \item A walk on graph $G$ is said to be closed if it starts and ends at the same vertex.
\end{itemize}

\noindent
If $T$ is a tree on graph $G$, then any walk on $T$ which visits its every edge exactly twice is called a tree traversal.
\end{definition}

We are now in position to formulate the general framework of the \textit{minimal spanning tree method} (or the \textit{MST method} for short):
\begin{description}
    \item[step 1.\hspace{0.25cm}] Given a complete weighted graph $G := (K_n,\omega)$ we find (one of) its minimal spanning tree(s) $T$.
    \item[step 2.\hspace{0.25cm}] Via a depth-first search (or DFS for short) on $T$ we construct a tree traversal (which depends on the root of the algorithm). 
    \item[step 3.\hspace{0.25cm}] We perform a shortcutting procedure on the tree traversal (from the previous step) to obtain a Hamiltonian cycle on $G$. 
\end{description}

The first two steps are widely recognized and can be found in numerous sources.\footnote{See \cite[Chapter 3.10]{Deo1974} or \cite[Chapter 12]{Papadimitriou1998} for the algorithms generating a minimal spanning tree, and \cite[Chapter 22.3]{Cormen2009} for the DFS algorithm.} Therefore, we will only discuss step 3, i.e., the shortcutting procedure, which is relatively less known.  

Let $G := (K_n,\omega)$ be a complete weighted graph on $n\in\N, n\geqslant 3$ vertices and let $T$ be a minimal spanning tree\footnote{We say that $T$ is ``\textit{a}'' rather than ``\textit{the}'' minimal spanning tree since for a given graph there might be more than one minimal spanning tree.} of $G$. We choose an arbitrary vertex $x_1\in V(T),$ which we refer to as the \textit{root of the tree}, and perform a DFS to obtain a tree traversal $(x_1,\ldots, x_{2n-1}).$ Next, we define $(y_1,\dots,y_n,y_{n+1})$ as the sequence obtained from $(x_1,\dots,x_{2n-1})$ by the \textit{shortcutting procedure}:
\begin{equation}\label{definition_of_yi}
    y_i:=\begin{cases}
    x_1, & i=1, n+1,\\
    x_{m(i)}, & i=2,\ldots,n,
    \end{cases}
\end{equation}

\noindent
where 
$$m(i):=\min\bigg\{j: i\leqslant j \leqslant 2n-1 \ \text{and} \ x_j \not\in \{x_1,\dots,x_{j-1} \} \bigg\}.$$ 

\noindent
At first glance, the definition (\ref{definition_of_yi}) may seem a bit daunting. However, there is a simple recipe for obtaining the shortcut $(y_1,\dots,y_n,y_{n+1}).$ We go through the elements of $(x_1,\dots,x_{2n-1})$ one by one and cross out every ``repetition'' (an element that has already appeared earlier in the sequence) except for $x_{2n-1}$ which is the same as $x_1$. In the ``shortcutted'' sequence $C_T := (y_i)_{i=1}^{n+1}$ every vertex (except for the root $x_1$) appears precisely once (after all, we did cross out all repetitions other than $x_1=x_{2n-1}$). It is easy to see that $C_T$ is in fact a Hamiltionian cycle on $G,$ which we call a \textit{$(T,x_1)-$shorcut} of the graph $G$.\footnote{We should bear in mind that there might be multiple shorcuts on a given graph - see Example \ref{example2}.} 

In order to reinforce our intuition, let us consider the following example of the shortcutting procedure:

\begin{example}\label{example2}
Consider a complete graph $K_6$ with vertices $V:=\{1,\dots,6\}$. Let $T$ be a spanning tree with the following set of edges 
\[
E(T)=\bigg\{\{1,2\},\{2,3\},\{3,4\},\{3,5\},\{5,6\}\bigg\}.
\]
\begin{figure}[H]
    \centering
    \includegraphics[width=0.4\textwidth]{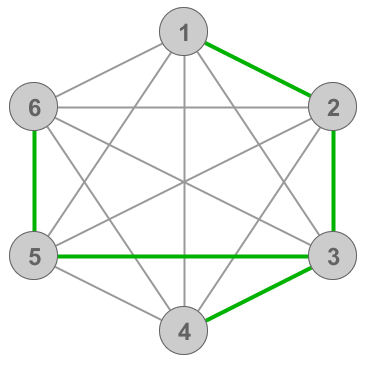}
    \caption{Spanning tree $T$ on the complete graph $K_6$.}
    \label{fig:K6}
\end{figure}

Let us consider a tree traversal starting at the vertex $1$ (the root of the tree $T$). We subsequently visit the vertices $2$ and $3$ - then we are forced to make a choice. Let us proceed with vertices $5$ and $6$. We then backtrack to $3$ and visit $4$. After retracing our steps back to $1$ we obtain a tree traversal $R_1$:
\[
R_1:=(1,2,3,5,6,5,3,4,3,2,1).
\]
Performing the shortcutting procedure we remove those vertices which appear multiple times (except for the final one) in the sequence $R_1.$ The final result is the Hamiltonian cycle \[C_1:=(1,2,3,5,6,4,1).\]

\begin{figure}[H]
    \centering
    \includegraphics[width=0.4\textwidth]{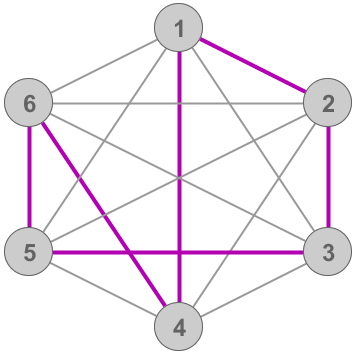}
    \caption{Hamiltonian cycle $C_1$.}
    \label{fig:K6C1}
\end{figure}

In order to see that $C_1$ is not the only possible $(T,1)-$shortcut, we note that it not mandatory to choose $5$ after $3$. We may as well choose $4$, backtrack to $3$, and only then go to $5$ and $6$ This results in a tree traversal 
\[
R_2:=(1,2,3,4,3,5,6,5,3,2,1).
\]
Applying shortcutting procedure to $R_2$ yields a different Hamilton cycle 
\[C_2:=(1,2,3,4,5,6,1).\]

\begin{figure}[H]
    \centering
    \includegraphics[width=0.4\textwidth]{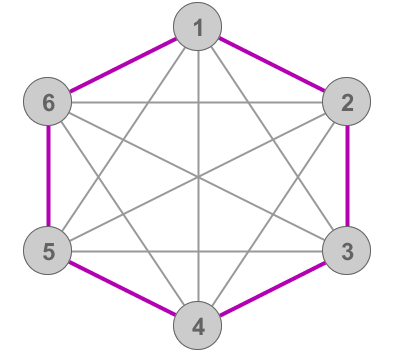}
    \caption{The Hamiltonian cycle $C_2$.}
    \label{fig:K6C2}
\end{figure}
\end{example}

\section{Performance of the MST method}

We already know that the MST method returns a Hamiltonian cycle for a given graph $G$. It is high time we examined how close this cycle is to being a solution to the TSP. To this end we compare the weights of the following two subgraphs of $G$: the optimal solution to the TSP $C^*$ and the minimal spanning tree $T$. In fact, we have the following inequality:
\begin{equation}\label{eq:comparing optimal and MST}
\omega(T)\leqslant \omega(C^*).
\end{equation}
The proof of this fact can be shortened to a simple observation that removing a single edge from $C^*$ leaves us with some spanning tree $T^*$, whose cost is (by definition of MST) bounded from below by $\omega(T)$.

The theorem below addresses this issue under the assumption that the graph is metric:\footnote{Compare with \cite[Thm. 35.2]{Cormen2009}.}
\begin{theorem}
Let $(K_n,\omega)$ be a complete, metric graph, $T$ be its minimal spanning tree and let $x \in V(T)$. Every $(T,x)-$shortcut $S$ satisfies 
\[
\omega(S) \leqslant 2\omega(T) \leqslant 2\omega(C^*).
\]
\label{weightofDFSTcycle}
\end{theorem}
\begin{proof}
Let $(x_k)_{k=1}^{2n-1}$ be a tree traversal for $T$ and let $(y_k)_{k=1}^{n+1}$ denote the elements in the $(T,x)-$cycle $S$. For every $i\leqslant n$ we have
\begin{equation}
\label{metric_case_assumption}
\begin{split}
 d(y_{i},y_{i+1}) &= d(x_{m(i)},x_{m(i+1)}) \leqslant d(x_{m(i)},x_{m(i)+1}) + d(x_{m(i)+1},x_{m(i+1)})\\ 
 &\leqslant \ldots \leqslant \sum_{j=m(i)}^{m(i+1)-1} d(x_{j},x_{j+1}),
\end{split}
\end{equation}

\noindent
which implies
\[
\omega(S) = \sum_{i=1}^n d(y_i,y_{i+1}) \leqslant \sum_{i=1}^n \sum_{j=m(i)}^{m(i+1)-1} d(x_{j},x_{j+1}) \leqslant \sum_{k=1}^{2n-2} d(x_k,x_{k+1}) = 2\omega(T).
\]

\noindent
This concludes the proof.
\end{proof}

In the language of optimization theory, Theorem \ref{weightofDFSTcycle} says that the MST method constructs a $2-$optimal approximation for the TSP on a \textit{metric} graph. The fact that the graph is metric is crucial, as the triangle inequality is used (repeatedly) in (\ref{metric_case_assumption}). 

Our next goal is to prove a counterpart of Theorem \ref{weightofDFSTcycle} in the case of $\gamma-$polygon graphs.

\begin{theorem}
Let $(K_n,\omega)$ be a complete, $\gamma-$polygon graph and let $T$ be its minimal spanning tree rooted at vertex $x_1$. Every $(T,x_1)-$shortcut $S$ satisfies 
\begin{equation}
\omega(S) \leqslant 2\gamma \omega(T) \leqslant 2 \gamma \omega(C^*).
\label{estimate_with_2gamma}
\end{equation}
\label{weight_of_shortcut_gamma_polygon}
\end{theorem}
\begin{proof}
Again, as in Lemma \ref{weightofDFSTcycle} let $(x_k)_{k=1}^{2n-1}$ be a tree traversal for $T$ and let $S = (y_k)_{k=1}^{n+1}$ be its $(T,x_1)-$shortcut. Performing an analogous reasoning to (\ref{metric_case_assumption}) we have
\begin{equation}
\forall_{i=1,\ldots,n}\ d(y_{i},y_{i+1}) = d(x_{m(i)},x_{m(i+1)}) \stackrel{\eqref{polygonalinequality}}{\leqslant} \gamma \sum_{j=m(i)}^{m(i+1)-1} d(x_{j},x_{j+1}).
\label{bmetric_case_assumption}
\end{equation}

\noindent
Consequently, we obtain
\begin{gather*}
\omega(S) = \sum_{i=1}^{n} d(y_i,y_{i+1}) \stackrel{(\ref{bmetric_case_assumption})}{\leqslant} 
\gamma \sum_{i=1}^n \sum_{j=m(i)}^{m(i+1)-1} d(x_{j},x_{j+1})
\leqslant \gamma\sum_{k=1}^{2n-2} d(x_k,x_{k+1}) \leqslant 2\gamma\omega(T),
\end{gather*}

\noindent
which concludes the proof.
\end{proof}

In the language of optimization theory, Theorem \ref{weight_of_shortcut_gamma_polygon} says that our MST method constructs a $2\gamma-$optimal approximation for the TSP. It is natural to examine how the estimate (\ref{estimate_with_2gamma}) places amongst other results of this kind known in the literature. As far as we are aware, the idea of using $\gamma-$polygon structure of a given graph is a new insight and has not been explored thus far. Instead, researchers focused on $\beta-$metric structure and arrived at four approximation methods, which are most popular in the field. The first one is the $\frac{3\beta^2 + \beta}{2}-$optimal approximation method, which was devised by Andreae, Bandelt \cite{Andreae1995} in 1995. This method is nowadays only of historical value since it was improved (in 2001) by Andreae himself to $(\beta^2+\beta)$-approximation algorithm \cite{Andreae2001}. In 2000 Bender and Chekuri \cite{Bender2000} constructed a $4\beta$-optimal approximation method while B\"{o}ckenhauer, Hromkovi\v{c}, Klasing, Seibert and Unge came up with a $\frac{3\beta^2}{2}-$approximation method in 2002.\footnote{See \cite{Bockenhauer2002} for the original paper introducing the $\frac{3\beta^2}{2}-$optimal approximation method, as well as \cite{Krug2013} where the author included a detailed discussion on the error bounds of the algorithm.} The discussion ``which approximation technique is better?'' can be best summarized by the following plot:
\begin{figure}[H]
    \centering
    \includegraphics[width=0.8\textwidth]{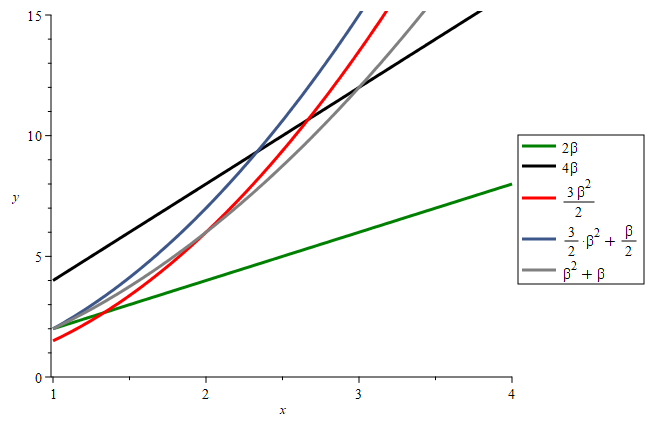}
    \caption{Plot depicting the curves of the approximation constants of the discussed methods.}
    \label{fig:K6C2}
\end{figure}
\noindent
For instance, one can see that the $\frac{3\beta^2}{2}-$optimal approximation method (the red curve) is worse than $(\beta^2 + \beta)-$optimal approximation method (drawn with light gray colour) for $\beta > 2$ but it is better for $\beta < 2$.

Where is the estimate (\ref{estimate_with_2gamma}) situated amongst these optimal approximations using the $\beta-$metric structure of a graph? In order to answer that question we need to have some kind of relation between $\beta$ and $\gamma.$ Obviously, as we remarked earlier we have $\gamma\geqslant \beta,$ but we also need some upper bound on $\gamma$ in terms of $\beta$. We focus our attention on a class of graphs where $\gamma \in [\beta,2\beta)$ and $\beta \geqslant 3.$ These conditions are not very restrictive and Example 1 provides an instance of a graph belonging to such a class. The crucial observation is the if $\gamma \in[\beta,2\beta)$ and $\beta\geqslant 3$ then
\[
2\beta \leqslant 2\gamma < 4\beta \leqslant \beta^2 + \beta < \frac{3}{2}\beta^2 < \frac{3\beta^2+\beta}{2}.
\]

\noindent
This means that there exists a large class of graphs for which our estimate (\ref{estimate_with_2gamma}) is better than anything known in the literature. Furthermore, the class we are discussing is not some ``secretive ensemble''. Given a weighted graph it suffices to compute its $\beta$ and $\gamma$ (which can be done relatively efficiently as proven by Theorem \ref{calculating_gamma}) and check two conditions: $\beta > 3$ and $\gamma < 2\beta$. If both of them are true then the application of estimate (\ref{estimate_with_2gamma}) is highly recommended. 

To close off the current section, we will show that the constant $2\gamma$ is in a sense ``the best possible'' constant in the MST method (unless we modify our approach in the next section). More formally, we show for any $\varepsilon>0$ there exists a complete, weighted graph $(K_n,\omega_n)$ such that
\[
\omega_n(S)>(2\gamma-\varepsilon) \omega_n(T).
\]

\begin{example}
Fix $\gamma\geqslant 1$ and let $(K_n,\omega_n)$ be a complete graph consisting of $n$ distinct vertices $x_1,\dots,x_{n}$, where 
\[
\omega_{n}\left(\{x_i,x_j\}\right):=\begin{cases}
1,& i=1, \ j>1\\
2\gamma,& i,j>1,\ i\neq j.
\end{cases}
\]
It is easy to see that there exists exactly one minimal spanning tree $T$ in such a graph, namely the one with 
\[
E(T)=\bigg\{(x_1,x_k)\  : \ k\in \{2,\dots, n\} \bigg\}.
\]
Since the weight of each edge of $T$ is equal to $1$, we have $\omega_{n}(T)=n-1$.

Performing a depth-first search algorithm on $T$ (starting at $x_1$) we obtain a tree traversal:
\[
W:=(x_1,x_2,x_1,x_3,x_1,\dots,x_n,x_1).
\]
Next, the shortcutting procedure yields a $(T,x_1)-$shortcut 
\[
S:=(x_1,x_2,x_3,\dots,x_n,x_1)
\]
with
\[
\omega_{n}(S)=1+(n-2)\cdot 2\gamma+1 = 2\cdot(1+\gamma(n-2)).
\] 
\noindent
It remains to note that 
\[
\lim_{n\rightarrow\infty}\ \frac{\omega_{n}(S)}{\omega_{n}(T)}=\lim_{n\rightarrow\infty}\ \frac{2\cdot(1+\gamma(n-2))}{(n-1)} = 2\gamma.
\]
\end{example}

\section{Christofides algorithm on $\gamma-$polygon graphs}
\label{christofides section}

In the previous section we laid out a $2\gamma-$approximation method for the TSP (constructing a $(T,x)-$-shortcut from a MST $T$ and the root $x$) and one may wonder whether we can do any better than that? Luckily, the answer is ``yes!'' and our final section of the paper is devoted to presenting the details of this refined approach. 

The core idea of our method comes from the work of Christofides, who came up with a $\frac{3}{2}-$optimal approximation method for the class of metric graphs.\footnote{See \cite{Christofides}.} Our goal in the current section is to extend the scope of his method to the class of $\gamma$-polygon graphs, which is (as far as we are aware) a novelty in the field.

To begin with, we recall a couple of definitions, which are vital in understanding what follows:\footnote{For the source of those definitions, we refer the Reader to the book of Diestel \cite[Chapters 1.-2.]{Diestel2000}.}

\begin{definition}
Let $(V,E)$ be a graph. The number of edges incident to $x\in V$, i.e. 
$$\deg(x):=\left|\left\{ \{x,y\} \in E \ : \ y\in V \right\}\right|,$$

\noindent 
is called the degree of vertex $x.$
\end{definition}

\begin{definition}
Let $(V,E)$ be a graph. A set of edges $E_M\subset E$ is called a matching, if no two edges of $E_M$ share a common vertex, i.e. for any two edges $\{x_1,y_1\},\{x_2,y_2\}\in E_M$ we have
\[
\{x_1,y_1\}\cap \{x_2,y_2\} =\emptyset.
\]
If each vertex has an incident edge in $E_M$, i.e.
\[
\forall_{x\in V} \ \exists_{y\in V} \ \{x,y\}\in E_M,
\]
then the matching $E_M$ is said to be a perfect matching.
\end{definition}

Of course every matching can be viewed as a subgraph of the original graph $(V,E)$-- $E_M$ is the ``new'' set of edges and the endpoints of these edges form the ``new'' set of vertices. The last definition that we need for the sequel is the following:

\begin{definition}
Let $(G,\omega)$ be a weighted graph and let $M\subset G$ be its perfect matching. We say that $M$ is a minimum-weight perfect matching if for any other perfect matching $M'\subset G$ we have $\omega(M)\leqslant \omega(M')$.
\end{definition}

At this point we have gathered all the necessary concepts to revisit the approximate solution for the TSP devised by Christofides. Given a complete weighted graph $(K_n,\omega)$ we follow the steps:

\begin{description}
        \item[step 1.\hspace{0.25cm}] Find the minimal spanning tree $T$ of $G$ (using Boruvka's, Prim's or Kruskal's algorithm for instance).
    
        \item[step 2.\hspace{0.25cm}] Let $V_{odd}$ be the set of all odd vertices (i.e., vertices with odd degrees) of $T$. Let $F$ be a weighted subgraph induced on $G$ by $V_{odd}$.
    
        \item[step 3.\hspace{0.25cm}] Find a minimum-weight perfect matching\footnote{Such matching exists because $F$ is a complete graph with even number of vertices. The latter observation is due to the ``handshaking lemma'' -- see Proposition 1.2.1 in \cite{Diestel2000}.} in $F$ and denote it by $M$. This can be done by applying the original Edmonds' blossom algorithm or one of its subsequent versions.\footnote{The initial version of minimum-weight perfect matching algorithm \cite{Edmonds1965} had complexity of order $\mathcal{O} \left(|E|\cdot |V|^2\right)$. This bound has been consistently improved over the years -- see Tables I and II in \cite{Cook1999} for a detailed exposition.}
    
    \item[step 4.\hspace{0.25cm}] Find an Eulerian walk\footnote{Just as Hamiltonian cycle is a cycle which visits every vertex exactly once, the \textit{Eulerian walk/cycle} (also \textit{circuit}) is a walk which traverses through each edge of the graph exactly once. This walk can be found using either Fleury's algorithm or Hierholzer's algorithm (with time-complexities  $\mathcal{O}(|E|^2)$ and $\mathcal{O}(|E|),$ respectively).}  $W$ in the multigraph $\GG = (V(T), \EE),$ where 
    $$\EE = \bigg\{(0,e) \ :\ e\in E(T)\bigg\} \cup \bigg\{(1,e) \ :\ e\in E(M)\bigg\}.$$


    \item[step 5.\hspace{0.25cm}] Perform shortcutting procedure on $W$, thus obtaining a Hamiltonian cycle $C$.
\end{description}

It might not be obvious at first glance why the multigraph $\GG$ constructed in fourth point of this algorithm necessarily has an Eulerian walk. These doubts are dispelled by the following lemma:\footnote{See Theorem 3.1 in \cite{balakrishnan1997}, Theorem 1.8.1 in \cite{Diestel2000} or Section 4.4 in \cite{Levin2015}.}

\begin{lemma}
If each vertex of a connected (multi-)graph has an even degree, then there exists an Eulerian cycle in this (multi-)graph.
\end{lemma}

As far as time complexity of Christofides' algorithm is concerned, we remark that in the worst case scenario we have to look for a minimum-weight perfect matching in the whole graph $G$ (when performing point 3.). This is arguably the choke-point of the whole procedure. The minimum-weight perfect matching can be found with the use of Gabow's algorithm\footnote{See \cite{Gabow1974,Gabow1990}.} with time complexity $\mathcal{O}(|V(G)|^3),$ although there are more efficient methods if all weights are natural numbers (which we do not assume in this paper). The time complexity of Gabow's algorithm also sets the complexity for the whole procedure. 

In the case of metric graphs, Christofides proved\footnote{See \cite[Thm. 1]{Christofides}.} that the Hamiltonian cycle $C$ returned by his algorithm satisfies $\omega(C)\leqslant \omega(T)+\omega(M)$. Next, Christofides established the inequality 
\begin{equation}\label{estimateonomegaM}
    \omega(M)\leqslant \frac{1}{2}\omega(C^*)
\end{equation} 

\noindent
where $C^*$ is an optimal solution for the TSP. Finally, using the inequality $\omega(T)\leqslant \omega(C^*)$, he arrived at the conclusion that $C$ is a $\frac{3}{2}$-optimal approximation. 

Let us emphasize that Chistofides' reasoning (and the constant $\frac{3}{2}$) hinges upon the metricity of the graph. Our objective now is to extend this method to a more general setting of $\gamma-$polygon graphs. We carry out the first 3 steps exactly as we did in the metric case. At the end of step 3 we have constructed a minimal-weight perfect matching, but \eqref{estimateonomegaM} may not hold since the graph need no longer be metric! So what is the couterpart of bound \eqref{estimateonomegaM} in the $\gamma-$polygonal case? The answer is given by the following lemma:\footnote{In particular, the inequality \eqref{estimateonomegaM} is a special case of Lemma \ref{lemma7} when $\gamma=1$.}

\begin{lemma}\label{lemma7}
Let $C^*$ be the optimal solution for TSP in a complete weighted $\gamma-$polygon graph $G$. Let $F$ be any complete subgraph of $G$ having an even number of vertices. If $M$ is a minimum-weight perfect matching in $F$, then 
$$\omega(M)\leqslant \frac{\gamma}{2}\omega(C^*).$$
\end{lemma}
\begin{proof}
Let $C^*:=(y_1,\dots,y_n,y_1)$, where $y_1$ is some arbitrarily chosen vertex from $F$. Let $z_1:=y_1$ and $C_F :=(z_1,z_2,\ldots,z_m,z_1)$ be an unique enumeration of vertices of $F$ such that $z_i$ preceeds $z_j$ in $C^*$ whenever $i<j$.\footnote{Informally, we follow the route of the Traveling Salesperson and label subsequent vertices of $F$ in the order of visiting.}

Notice that 
\begin{equation}\label{mincostmatchingcost}
    \forall_{i=1,\ldots,m}\ \omega(\{z_i,z_{i+1}\}) \leqslant \gamma\sum_{j=k_i}^{k_{i+1}-1} \omega(\{y_j,y_{j+1} \}),
\end{equation}
where $y_{k_i}=x_i$ and $y_{k_{i+1}}=x_{i+1}$ (and $x_{m+1} := x_1$). This leads us to the conclusion that 
\begin{equation}\label{estimateonCF}
    \omega(C_F)\leqslant \gamma \omega(C^*).
\end{equation}

Next, we split the cycle $C_F$ into two matchings by ``alternating the edges'', i.e. 
\begin{equation*}
    \begin{split}
    M_1&:=\left(V(F),\left\{ \{z_1,z_2\}, \{z_3,z_4\}, \dots, \{z_{m-1},z_m\}  \right\}\right),\\
    M_2&:=\left( V(F),\left\{ \{z_2,z_3\}, \{z_4,z_5\}, \dots, \{z_m,z_1\}  \right\}\right).
    \end{split}
\end{equation*}

\noindent
Obviously we have $\omega(M_1)+\omega(M_2)=\omega(C_F),$ which implies that either $\omega(M_1)\leqslant \frac{\omega(C_F)}{2}$ or $\omega(M_2)\leqslant \frac{\omega(C_F)}{2}$. Either way, if $M$ is the minimum-weight perfect matching  then 
\[
\omega(M)\leqslant \min\left\{ \omega(M_1),\omega(M_2) \right\} \leqslant \frac{\omega(C_F)}{2} \stackrel{(\ref{estimateonCF})}{\leqslant} \frac{\gamma}{2} \omega(C^*),
\]
\noindent
which concludes the proof.
\end{proof}


Next, we proceed with a counterpart of step 4. and mildly modify the Hierholzer algorithm to guarantee that the first edge of the resulting Eulerian cycle is specifically taken from the matching. Let $\mathbb{G} := ((V,\mathbb{E}),\omega)$ be a weighted multigraph, where 
\begin{itemize}
    \item $V(\mathbb{G}) := V(G) = V(T)$;
    \item $\mathbb{E}(G) := \left\{ (0,e) \ : \ e\in E(T)   \right\} \cup \left\{ (1,e) \ : \ e\in E(M) \right\}.$
\end{itemize}

\noindent 
Our modified variant of Hierholzer algorithm can be described in the following points:

\begin{enumerate}
    \item Select arbitrary starting vertex $x_1\in V(\mathbb{G})$, which is incident to some edge from $E(M)$. Let $W:=(y_1)$, where $y_1:=x_1$. Mark $y_1$ as a ``\textit{recently visited vertex}''. 
    \item Extend the sequence of vertices $W$ in the following way: 
    \begin{enumerate}
        \item select any neighbour $y$ of the ``\textit{recently visited vertex}'' connected to it by an ``\textit{unused}'' edge in $E(M)$ (if possible) or in $E(\mathbb{G})$ (this is always possible because each vertex in $V(\mathbb{G})$ has even degree).\footnote{Due to this ``prioritization'' $\{y_1,y\}$ is guaranteed to be taken from $E(M)$.}
        \item add $y$ to the sequence $W$. Mark $y$ as the ``\textit{recently visited vertex}'' and denote the traversed edge as ``\textit{used}''. If there is an unused edge incident to $y$, go back to step a).\footnote{
        At each step of the algorithm the only vertices with odd number of unused edges are $x_1$ and the current vertex $y$. Therefore, the loop can terminate only if we return to the initial vertex $x_1$.}
        
    \end{enumerate}
    \item If there are any unused edges after this process, start at any vertex $x\in W$ which has at least one neighbour not in $W$. Repeat the procedure described in step 2 obtaining a closed walk $W_x.$ Replace the last appearance of $x$ in sequence $W$ with $W_x$.
\end{enumerate}

Having obtained an Eulerian walk $W:=(x_1,\dots,x_{k+1}),\ k=|E(M)|+|E(T)|$ by the means of previously discussed algorithm, we can move to the last, fifth step of our procedure, i.e., shortcutting. This is a rather subtle point. If we were to shortcut in a naive way (as we did in the metric case) then we could cross out the edges from the matching. This, in turn, would cause the error estimate to include the relaxation constant $\gamma$ twice -- once in the shortcutting procedure and the second time in the estimation of the total cost of the matching. Since $\gamma$ usually exceeds $1$, we definitely want to avoid having $\gamma^2$ in the error estimate.

In order to avoid that obstacle we introduce the following enhanced shortcutting procedure, which results in the Hamiltonian cycle $\tilde{C}_T:=(y_1,\dots,y_n,y_{n+1})$:

\begin{enumerate}
    \item Assign $y_1:=x_1$ and $y_2:=x_2$. From the construction of $W$ (the Eulerian walk from the previous step) it follows that $\{x_1,x_{2}\}\in E(M)$. Let $i:=3$ and $j:=3$.
    \item While $i\leqslant n$ perform the following steps:
        \begin{enumerate}
            \item If $x_j\notin V(M)$ and it already has appeared in $\tilde{C}_T$, increment $j$. 
            
            \item If $x_j\notin V(M)$ has not appeared in $\tilde{C}_T$ so far, put $y_i:=x_j$. In such case, increment both $i$ and $j$.
            
            \item If $x_j\in V(M)$, $\{x_j,x_{j+1}\}\notin E(M)$ and $\{x_{j-1},x_j\}\notin E(M)$, increment $j$. 
            
            \item Otherwise, i.e., in the situation where $x_j\in V(M)$ and either $\{x_j,x_{j+1}\}\in E(M)$ or $\{x_{j-1},x_j\}\in E(M)$, let $y_i:=x_j$. Increment both $i$ and $j$.
        \end{enumerate}
\end{enumerate}

Notice that these four situations described above not only are mutually exclusive, but they also guarantee that each vertex will  appear exactly once in $\tilde{C}_T$. In the case where $x_j\notin V(M)$, the assignment takes place upon the first encounter of the discussed vertex. In the case where $x_j\in V(M)$, the spot in which $x_j$ is assigned is unique as well, due to the fact that every vertex in $V(M)$ has a degree in subgraph $M$ equal to $1$. This way the resulting Hamiltonian cycle $\tilde{C}_T$ contains all edges from $E(M)$ i.e., no edge from $E(M)$ was erased during the shortcutting procedure. 

\begin{theorem}
We have 
$$\sum_{e \in E(\tilde{C}_T)\backslash E(M)}\ \omega(e) \leqslant \gamma \omega(T).$$
\label{greatheorem}
\end{theorem}
\begin{proof}
If $e \in E(\tilde{C}_T)\backslash E(M)$ then $e = \{y_k,y_{k+1}\}$ for some $y_k,y_{k+1}\in V(G).$ Furthermore, there exist
\begin{equation}\label{chain_of_vertices}
x_{k_i} := y_k,\ x_{k_i + 1},\ \ldots,\ x_{k_j} := y_{k+1}
\end{equation}

\noindent
such that 
$$\{x_{k_i}, x_{k_i+1}\}, \ldots, \{x_{k_j-1}, x_{k_j}\} \in (E(M)\cup E(T)),$$

\noindent
because all the $x_k$'s are the vertices from the Eulerian walk $W$. Next, we observe that 
$$\forall_{l = 0,\ldots,k_j-k_i-1}\ \{x_{k_i+l}, x_{k_i+l+1}\} \not\in E(M),$$

\noindent
since that would imply that an edge from the perfect matching $M$ was shortcutted, which is impossible.\footnote{If $k_j=k_i+1$ then the discussed chain of vertices has length 2 and no edge was subjected to the shortcutting procedure. This means that $e=\{y_k,y_{k+1}\}$ belongs to $\tilde{C}_T$ and from the assumption $e\notin E(M)$.}
Thus (by the $\gamma-$polygon inequality) we have establish that
\begin{equation}\label{nierownosc_dla_jednej_krawedzi}
\forall_{e \in E(\tilde{C}_T)\backslash E(M)}\ \omega(e) \leqslant \gamma \sum_{l=0}^{k_j-k_i-1}\ \omega(\{x_{k_i+l}, x_{k_i+l+1}\}).
\end{equation}

Observe that for distinct $i,j$ we have that chains of edges defined by \eqref{chain_of_vertices} are disjoint (this is due to the fact that Eulerian walk $W$ visits every edge in multigraph $\mathbb{G}$ exactly once). As a consequence, on the right-hand side of the inequalities \eqref{nierownosc_dla_jednej_krawedzi} no edge appears more than once for all $e\in E(\tilde{C}_T)\setminus E(M)$. Finally, we have
\[
\sum_{\{y_k,y_{k+1}\} \in E(\tilde{C}_T)\backslash E(M)}\ \omega(\{y_k,y_{k+1}\}) \leqslant \gamma \sum_{\{k_i\ : \ y_k=x_{k_i}\}} \sum_{l=0}^{k_j-k_i-1}\ \omega(\{x_{k_i+l}, x_{k_i+l+1}\}) \leqslant \gamma E(T),
\]

\noindent
which ends the proof.
\end{proof}

Since $E(M)\subset E(\tilde{C}_T)$ then 
\[
\omega(\tilde{C}_T) =\sum_{e \in E(\tilde{C}_T)\backslash E(M)}\ \omega(e)  + \sum_{e\in E(M)}\ \omega(e) \stackrel{Theorem\ \ref{greatheorem},\ Lemma\ \ref{lemma7}}{\leqslant} \gamma \omega(T) + \frac{\gamma}{2}\omega(C^*) \leqslant \frac{3\gamma}{2}\omega(C^*).
\]



Thus we have reached the climax of our paper:

\begin{theorem}
Let $(K_n,\omega)$ be a complete  $\gamma$-polygon graph. There exists a $\frac{3\gamma}{2}$-approximation to the Travelling Salesperson Problem which can be computed in $\mathcal{O}(n^3)$ time.
\end{theorem}

\end{document}